\documentclass{amsart}
\usepackage{amssymb, latexsym, amsmath, eucal,cite,amsthm}
\usepackage{graphicx}
\usepackage{wrapfig}

\usepackage{tikz}
\usepackage{caption}
\usepackage[labelformat=simple,labelfont={}]{subcaption}

\usepackage{setspace}
\usepackage[letterpaper, left=2.5cm, right=2.5cm, top = 1in, bottom = 1in]{geometry}
\theoremstyle{plain}
\newtheorem{thm}{Theorem}[section]
\newtheorem{lemma}[thm]{Lemma}
\newtheorem{corollary}[thm]{Corollary}

\usepackage{times}
\usepackage[italic,defaultmathsizes]{mathastext}

\theoremstyle{definition}

\def\ca{\lower1.3em\hbox{\includegraphics{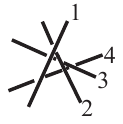}}}
\def\cb{\lower1.3em\hbox{\includegraphics{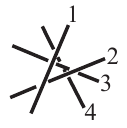}}}
\def\cd{\lower0.9em\hbox{\includegraphics{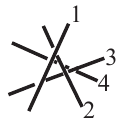}}}
\def\ce{\lower0.9em\hbox{\includegraphics{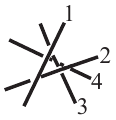}}}
\def\cf{\lower0.9em\hbox{\includegraphics{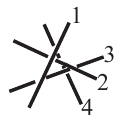}}}
\def\cg{\lower0.9em\hbox{\includegraphics{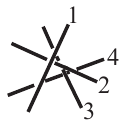}}}

\def\statea{\lower1.1em\hbox{\includegraphics{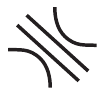}}}
\def\stateb{\lower1.1em\hbox{\includegraphics{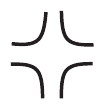}}}
\def\statec{\lower.9em\hbox{\includegraphics{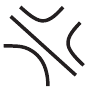}}}
\def\stated{\lower.9em\hbox{\includegraphics{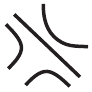}}}
\def\statee{\lower.9em\hbox{\includegraphics{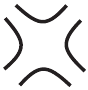}}}
\def\statef{\lower.9em\hbox{\includegraphics{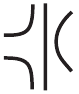}}}
\def\stateg{\lower.8em\hbox{\includegraphics{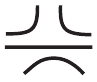}}}
\def\stateh{\lower.9em\hbox{\includegraphics{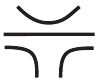}}}
\def\statei{\lower.8em\hbox{\includegraphics{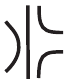}}}
\def\statej{\lower.9em\hbox{\includegraphics{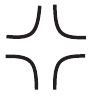}}}
\def\statek{\lower.9em\hbox{\includegraphics{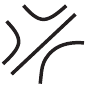}}}
\def\statel{\lower.9em\hbox{\includegraphics{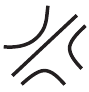}}}
\def\statem{\lower.8em\hbox{\includegraphics{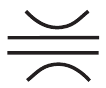}}}
\def\staten{\lower1.1em\hbox{\includegraphics{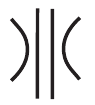}}}
\def\stateo{\lower.9em\hbox{\includegraphics{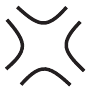}}}
\def\statep{\lower.9em\hbox{\includegraphics{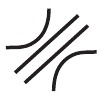}}}

\numberwithin{equation}{section}

\begin{document}
\title[Quadruple Crossing Number of Knots and Links]
{Quadruple Crossing Number of Knots and Links}

\date{\today}
\author[Colin Adams]{Colin Adams}
\address{Department of Mathematics and Statistics, Williams College, Williamstown, MA 01267}
\email{Colin.C.Adams@williams.edu}

\begin{abstract} A quadruple crossing is a crossing in a projection of a knot or link that has four strands of the knot passing straight through it.  A quadruple crossing projection is a projection such that all of the crossings are quadruple crossings. In a previous paper, it was proved that every knot and link has a quadruple crossing projection and hence, every knot has a minimal quadruple crossing number $c_4(K)$. In this paper, we investigate quadruple crossing number,  and in particular, use the span of the bracket polynomial to determine quadruple crossing number for a variety of knots and links.
\end{abstract}

\maketitle

\section{Introduction}\label{S:intro} In \cite{Adams}, a multi-crossing (also called an $n$-crossing) of a knot is defined to be a singularity in a projection such that n strands of the knot cross straight through the singular point.   In that paper, it was proved that for every fixed $n \geq 2$, any given knot or link has a projection such that all singularities are $n$-crossings. Hence, we can define $c_n(K)$  to be the minimal number of $n$-crossings in any such projection.  In that paper, triple crossing projections were considered, and $c_3(K)$ was determined for a variety of knots and links. 

In this paper, we continue the investigation, extending results to quadruple crossing projections of knots and links. For simplicity, define $q(K) = c_4(K)$. A quadruple crossing has strands at four heights, which we label top to bottom as 1,2,3,and 4 respectively. Reading clockwise around the crossing, and always starting with the top strand, there are six types of crossings, denoted $c_{1234}, c_{1243}, c_{1324}, c_{1342}, c_{1423}$,  and $c_{1432}$. Note that if we reflect the projection of one of these crossings in a line perpendicular to the strand labelled 1, we find $c_{1234}$ paired with $c_{1432}$, $c_{1243}$ paired with $c_{1342}$ and $c_{1324}$ paired with 
$c_{1423}$. This pairing will be helpful later.

\begin{figure}[h]
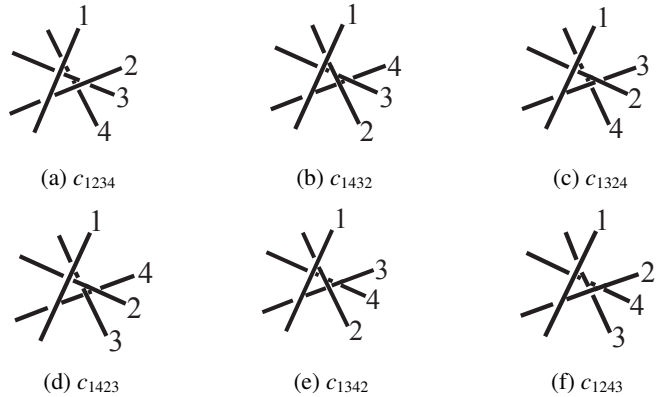

	\begin{subfigure}[b]{.2\textwidth}
		\centering
		{\includegraphics[height=20mm]{1234new.pdf}}
		\label{fig:1234}
		\caption{$c_{1234}$}
	\end{subfigure}
	\begin{subfigure}[b]{.2\textwidth}
		\centering
		{\includegraphics[height=20mm]{1432new.pdf}}
		\caption{$c_{1432}$}
		\label{fig:1432}
	\end{subfigure}
	 \begin{subfigure}[b]{.2\textwidth}
		\centering
		{\includegraphics[height=20mm]{1324new.pdf}}
		\label{fig:1324}
		\caption{$c_{1324}$}
	\end{subfigure} \\ \vspace{.5cm}
		\begin{subfigure}[b]{.2\textwidth}
		\centering
		{\includegraphics[height=20mm]{1423new.pdf}}
		\caption{$c_{1423}$}
		\label{fig:1423}
	\end{subfigure}\begin{subfigure}[b]{.2\textwidth}
		\centering
		{\includegraphics[height=20mm]{1342new.pdf}}
		\label{fig:1342}
		\caption{$c_{1342}$}
	\end{subfigure}	
	\begin{subfigure}[b]{.2\textwidth}
		\centering
		{\includegraphics[height=20mm]{1243new.pdf}}
		\caption{$c_{1243}$}
		\label{fig:1243}
	\end{subfigure}

	\vspace{.5cm}
	\caption{Quadruple crossings.}

\label{fig:quadcrossings}
\end{figure}
\vspace{.5cm}

Note that a singe quadruple crossing resolves into six double crossings, and hence, $q(K) \geq \frac{c(K)}{6}$ where $c(K)$ is the traditional crossing number. This bound is realized by a $(4k, 4)$-torus link $L$, which has four components, and is known to have crossing number $12k$.  Each meridianal half-twist is realized by a quadruple crossing, so $q(L) = 2k$. 

    In Section 2, we consider further the relation between $q(K)$ and $c(K)$.  In \cite{Adams}, it was proved that $c_3(K) \leq c(K) -2$ for all knots and links except 2-braid knots, and for 2-braid knots, $c_3(K) \leq c(K) -1.$ One might hope that $c(K) > c_3(K) > c_4(K)$, however, this is not the case. For example,  the figure-eight knot $J$ has $c_3(J)=c_4(J)=2.$ One might also hope that at least $c_n(K) \geq c_{n+1}(K)$ for all $k\geq 2$. At this point, we do not know this. We do know that $c_n(K) \geq c_{n+2}(K)$ for all $k\geq 2$. In Section 2, we  prove that $c_4(K) < c(K)$, the first proof of which was given by Michael Landry.  This section is independent of the subsequent sections.
     
 In Section 3, we relate the quadruple crossing number of a knot or link to the span of the bracket polynomial. In  \cite{Adams}, it was shown that $span(<K>) \leq 8c_3(K)$. Here, we show that span$(<K>) \leq 16 q(K)$. From this, it follows that if $K$ is an alternating knot or link, we obtain a stronger lower bound $q(K) \geq \frac{c(K)}{4}$. 
 
 In Section 4, after introducing some moves to turn a traditional projection into a quadruple crossing projection, we use this result to determine the quadruple crossing number of a variety of knots.We include a table of specific quadruple crossing number for many of the prime knots of 10 or fewer crossings.

Further investigations into multi-crossings appear in \cite{Ad}, where it was proved that every knot and link has a projection with just one multi-crossing.  The minimal $n$ for which there is a single $n$-crossing is called the \"ubercrossing number of the knot. One can further prove that for knots, there is such a projection that resembles a flower, called a petal projection, so that there are no nested loops leaving and re-entering the single multi-crossing. Hence, every knot has a petal number, which is the least $n$ for the single $n$-crossing of a petal projection. \"Ubercrossing number and petal number and bounds on them were determined for a variety of knots and links.

I would like to thank  B. DeMeo, M. Montee, S. Park, S. Venkatesh and F. Yhee for helpful conversations and especially Alex Lin, who determined the skein relations for quadruple crossings that appear below and Michael Landry, who first proved Theorem 2.1 below. 

\section{Upper bound on quadruple crossing number}

In this section, we prove that $q(K) < c(K)$ for all nontrivial knots and links. We first note that $q(K) \leq c(K)$. This was proved in \cite{Adams}, where the operation as in Figure \ref{quadruple} allows one to turn every double crossing into a quadruple crossing. The obvious generalization immediately implies $c_n(K) \geq c_{n+2}(K)$ for all $k\geq 2$. 

\begin{figure}[h]
\begin{center}
\includegraphics[scale=0.7]{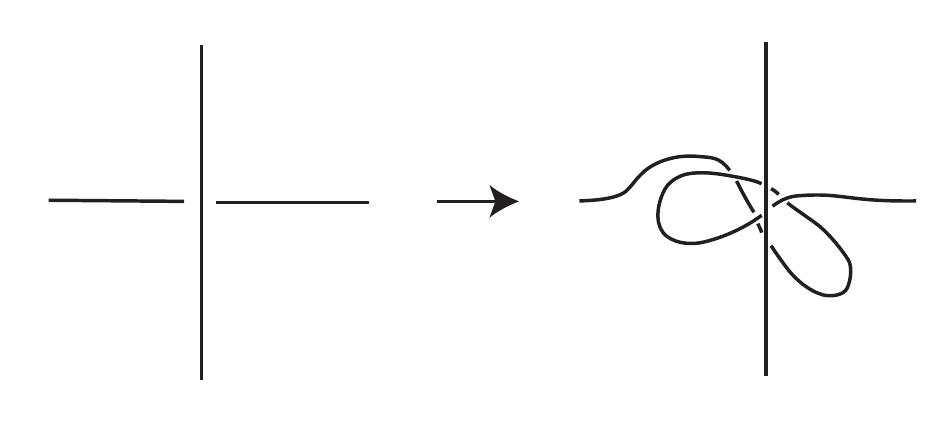}
\caption{Turning a double crossing into a quadruple crossing.}
\label{quadruple}
\end{center}
\end{figure}

\medskip

 We  note the following generalization of folding, the original version of which was also introduced in \cite{Adams}. There, a closed non-self-intersecting loop in the projection plane was called a {\em crossing covering circle} if it intersected the knot projection only in crossings, and at each crossing, there were two strands of the knot coming out to each side of the loop. We call the length of the circle the number of crossings it passes through. Given such a circle that crosses an even number of crossings, we can perform a quadruple folding, as in Figure \ref{quadfolding}. First we take the overstrand at one of the crossings hit by the crossing covering circle and we stretch it around the circle as in Figure \ref{quadfolding}(b), eliminating one crossing in the process and turning the other crossings on the circle into triple crossings. Then we take the second strand making up the original crossing, and we stretch it around the circle also. But as we do so we push it alternately inside and outside the first stretched strand, switching as we pass through each crossing. We thus obtain a projection where the $n$ double crossings that were intersected by the crossing covering circle have now become $n-1$ quadruple crossings.

\begin{figure}[h]
\begin{center}
\includegraphics[scale=0.7]{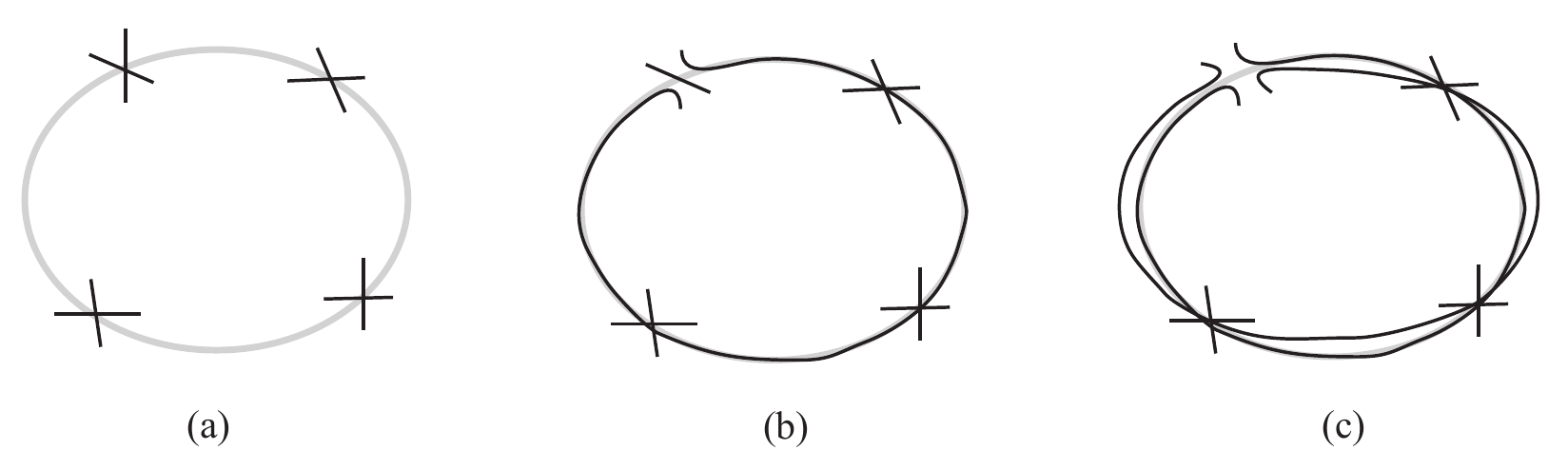}
\caption{Using an even length covering crossing circle to turn double crossings into quadruple crossings.}
\label{quadfolding}
\end{center}
\end{figure}

\begin{thm} $q(K) < c(K)$ for every nontrivial knot and link.
\end{thm}

\begin{proof} From what we have said, given a minimal crossing projection $P$ of $K$, it suffices to find a single crossing covering circle of even length, say $n$. Then by quadruple folding it, we obtain $n-1$ quadruple crossings, and the remainder of the double crossings can be converted directly to quadruple crossings as in Figure \ref{quadruple}, yielding fewer quadruple crossings than double crossings. 

If the projection contains a complementary region with an even number of edges, then we can take our crossing covering circle to surround it, passing through each of its crossings, and we are done. So we may assume that all of the complementary regions of $P$ have an odd number of edges.

Choose a crossing $x$. Replace the projection with the corresponding planar graph. The four complementary regions at $x$ must be all distinct, since if they were not, the original projection would not have been reduced. Call the regions $A$, $B$, $C$ and $D$. Let 
$R= A \cup B \cup C \cup D$. If $P - R$ is connected, then we can form a loop that shadows the boundary of  $R$, passing alternately inside and outside the outer boundary of each of the complementary regions, as in Figure \ref{regions}. That circle will have even length. 

\begin{figure}[h]
\begin{center}
\includegraphics[scale=0.5]{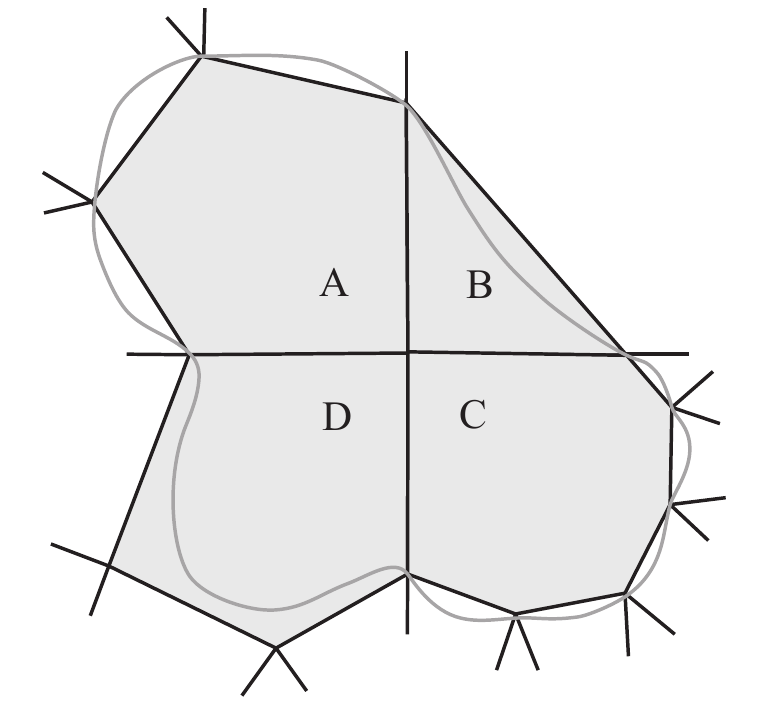}
\caption{Finding an even length crossing covering circle.}
\label{regions}
\end{center}
\end{figure}

If $P- R$ is not connected, then two of the regions must touch each other in more than one connected component. Suppose two opposite regions touch each other at a second crossing $y$. Then we can form a circle of length 2 that passes through these two regions and the crossings $x$ and $y$. If two opposite regions do not touch each other, say $A$ and $C$, we can take our covering crossing circle of even length as in Figure \ref{regions}, shadowing the boundary of $R$, to the outside on the boundaries of $A$ and $C$ and to the inside on the boundaries of $B$ and $D$.

\end{proof}

\section{Quadruple crossing number and the bracket polynomial}

In this section, we consider what the span of the bracket polynomial can tell us about quadruple crossing number. The bracket polynomial was introduced in \cite{Kau} and is defined by considering all the ways to split crossings in a fixed projection $P$ of $K$ and then taking $<P> = \Sigma \,A^{a(s)} A^{-b(s)} (-A^{2}- A^{-2})^{|s|-1}$, where the sum is over all states $s$ of $P$, $a(s)$ is the number of A-splits, $b(s)$ is the number of B-splits, and $|s|$ is the number of circles in the state $s$. Although the bracket polynomial depends on the particular projection, its span is an invariant for $K$.

Each quadruple crossing in a quadruple crossing projection will be one of the six possibilities mentioned above. When we resolve a single quadruple crossing into six double crossings and then split those crossings each in the two ways possible, we obtain 64 possible splittings. Many of these contain additional components, but once the components are removed, we are left with 14 possible splittings, four that consist of four parallel strands, called {\it parallel splits}, two that consist of four u-shaped components, called {\it U splits}, and eight that have two parallel components and two U components, called {\it mixed splits}. So the second line of the first skein relation below consists of one U split, two mixed splits and two parallel splits in that order.

\begin{thm} \label{skein}
\Large{

\begin{align}
 <\cb> =&  \quad A^{2}(<\statej> + <\staten >)  \\ &+ A^{0} (<\stateo> +<\statef> +<\statei> + \statea> + <\statep>) \notag \\&+ A^{-2} (<\statej> + <\statec> +<\stated> + <\statek> +<\statel>) \notag\\& + A^{-4} (<\stateo> + <\stateg> +<\stateh>) + A^{-6}(<\statem>) \notag
\end{align} 

\begin{align}
< \ca> =& \quad A^6<\statea>  + \quad A^4(< \stateb> +<\statec> +<\stated>)  \\ & + A^{2} (<\statee> + <\statef>+<\stateg> +<\stateh> + <\statei>) \notag \\ & + A^0(<\statej> +<\statek> + <\statel> +<\statem> +<\staten>)  \notag \\ & + A^{-2} (<\stateo> + <\statep>) \notag
\end{align}

\begin{align} <\cf> =&  \quad  A^{4}<\stateb >  \\ & + A^{2} (<\statef> +<\statei> + <\statep>+<\statea>)  \notag \\ & + A^{0} (<\stateb> + <\statec> +<\stated>  +<\statek> +<\statel> +<\staten>) \notag \\ & +A^{-2} (2<\statee> + < \stateg> +<\stateh>) + A^{-4} (<\statem>)  \notag
\end{align}

\begin{align} <\cg> =&  \quad A^{4}< \statea> +\quad  A^{2} (2<\stateb> +<\statec> + <\stated>)  \\ & + A^{0} (<\statee> + <\statef> +<\stateg> + <\stateh> +<\statei> +<\statep>) \notag \\ & +A^{-2} (<\statek> +<\statel> +<\statem> +<\staten> ) \notag \\ & + A^{-4} (<\statee>) \notag
\end{align}

\begin{align} <\cd> =&  \quad A^{4}<\statea >  \\ & + A^{2} (<\stateb> +<\statec> + <\stated> +<\staten>)\notag \\ & + A^{0} (2<\statee> + <\statef >+<\stateg >+<\stateh >+<\statei >) \notag \\ & +A^{-2} (<\stateb> + <\statek > + <\statel> +<\statem>) \notag \\ & + A^{-4} (<\statep> ) \notag
\end{align}

\begin{align} <\ce> =&  \quad A^{4}<\staten >  \\ & + A^{2} (<\statee> + <\statef> +<\statei> +<\statea>) \notag \\ & +
A^{0} (2<\stateb>  + <\statec> + <\stated> +<\statek> + <\statel> ) \notag \\ & +A^{-2} (<\statee> + < \stateg> +<\stateh> +<\statep>) \notag 
\\& + A^{-4} (<\statem> )  \notag
\end{align} 
 }

\end{thm}

\begin{proof} This follows immediately by resolving each of the particular quadruple crossings into its six double crossings, and then applying the skein relation for the bracket polynomial  applied to double crossings (cf. \cite{Kau}) to split each of the crossings,  eliminate the closed components and clean up the result. These relations were first determined by Alex Tong Lin.
\end{proof}

Given a crossing $c$, define the first level splits to be those splits corresponding to the highest power of $A$ in the skein relation for that crossing in Theorem \ref{skein}. The second level splits are defined to be those that correspond to a power of $A$ two less than those of the first level splits, the third level splits are those that correspond to a power of $A$ that is four less than those of the first level splits, etc. 

Given a particular splitting of a crossing that occurs in a state $s$, we can perform a {\it split move} which replaces two adjacent strands in the splitting by the two other strands that connect their endpoints as in Figure \ref{splitmove}. The resulting state $s'$ will satisfy $|s'| = |s| \pm 1$.   

\begin{figure}[h]
\begin{center}
\includegraphics[scale=1.0]{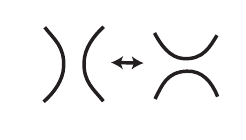}
\caption{Performing a split move on a splitting changes the resultant number of components by 1.}
\label{splitmove}
\end{center}
\end{figure}

Note that each split can be obtained from any other split by a sequence of such split moves. Each move changes the number of components in a state by $\pm 1$.

\begin{lemma} \label{splitmovelemma}Let $s$ be a particular state. 

\begin{enumerate}

\item Changing a split in $s$ of a given level to one of a lower level that is one split move different cannot increase the highest power of $A$ in the polynomial associated to $s$. 

\item Changing a split in $s$ of a given level to one of a higher level that is one split move different cannot decrease the lowest power of $A$ in the polynomial associated to $s$. 

\end{enumerate}
\end{lemma}

\begin{proof} In the first case, by the skein relations, the power of $A$ from the relation will go down by at least 2. But the number of components in the state increases by at most 1, and therefore the additional $(-A^2 -A^{-2})$ that comes from the potential increase in the number of components can at most offset the loss of 2 in the power of A. Hence the new state has greatest power no larger than the original. A similar argument holds for the second statement.  \end{proof}

\begin{thm} \label{span} For any nontrivial knot or non-splittable link, span$(<K>) \leq 16 q(K)$.
\end{thm}

\begin{proof}Let $P$ be a quadruple crossing projection of $K$ with $q$ triple crossings.  Each crossing is one of our six types. Define a state to be $s_{max}$ if it is obtained by splitting each crossing other than $c_{1342}$ or $c_{1243}$ as a first level split,  such that if there is more than one first level split (as occurs for $c_{1234}$), we choose a collection of first level splits that maximizes the number of component circles in the state. In the case of $c_{1342}$ or $c_{1243}$, we consider their parallel second level split. We use it rather than the first level split, if it intersects four components of the state or it intersects two components and the parallel fourth level split intersects four components. We call each of the splittings used to form $s_{max}$ a {\it high splitting}.

Similarly, define the state $s_{min}$ to be the state obtained by splitting each crossing other than $c_{1342}$ or $c_{1243}$ as a fifth level split, such that if there is a choice of fifth level splits (as occurs for $c_{1432}$), we choose a collection that maximizes the number of circles in $s_{min}$. For $c_{1342}$ and $c_{1243}$, we consider their parallel fourth level split. We use it rather than the fifth level split if we used the second level split for $s_{max}$. We call each of the splittings used to form $s_{min}$ a {\it low splitting}.

We first prove that there is no state with a higher power of $A$ than $s_{max}$. Any other state is obtained from  $s_{max}$ by changing some number of the first level splits on the crossings other than $c_{1342}$ or $c_{1243}$ . By checking the four skein relations for the four other crossings  in Theorem \ref{skein}, we see that each of the second level splits is obtained from one of the first level splits by a single split move. Each third level split is obtained from a second level split by a split move. And so on. Similarly, each fourth level split is obtained from a fifth level split by a single split move, etc.  

Thus, in the cases of $c_{1432}, c_{1324}$ and $c_{1423}$, where there is only one first level splitting, Lemma \ref{splitmovelemma} implies that changing the first level split  to a second level split results in a  new state that has greatest power no larger than the original. 

In the case of $c_{1234}$, where there are two first level splittings, each second level splitting is one split move away from one of the first level splittings. If we change the first level splitting in $s_{max}$ to a second level splitting and that splitting differs from the first one by one split move, then the number of components can go up by at most one, and the greatest power of $A$ is unchanged. If we change the first level splitting to a second level splitting and the number of components goes up by at least 2, then the other first level splitting, which differs from the second level splitting by only one split move, would have generated more components for $s_{max}$ than did the one we had, a contradiction to our procedure for choosing the splittings that form $s_{max}$. 

    Finally, we consider the pair of crossings $c_{1243}$ and $c_{1342}$.  Except for the parallel second level split, all other second level splits differ from the first level split by one split move, and hence Lemma \ref{splitmovelemma} shows that we cannot increase the highest power of $A$ by switching to one of those. The same holds for the third, fourth and fifth level splits. Suppose now that the parallel second level split is used in $s_{max}$. Then there are two cases.  It could be that the  parallel second level split generates four components, which implies that the the first level split and the fifth level split each generate one component and the parallel fourth level split generates two.  Or it could be that the parallel fourth level split generates four components, while the parallel second level split generates two components and the first level and fifth level splits each generate one component.  In either case, changing this split cannot yield a higher power of $A$, since in the first case, the highest power of $A$ in the polynomial term of the parallel second level split beats that of the first level split by $A^4$ and in the second case, the highest power of $A$ in the polynomial term of the parallel second level split ties that of the first level split,  and no other split has a  higher power of $A$ in its polynomial term than the first level split.
    
 The same type of arguments apply to show that the polynomial term of $s_{min}$ possesses the lowest power of $A$.

Given a projection $P$, define $M_P$ to be the highest exponent of $A$ in the polynomial term associated to $s_{max}$ and  $m_P$ to be the lowest exponent of $A$ in the polynomial term associated to $s_{min}$.

Let $|c_{ijk}|$ be the number of crossings of type $c_{ijk}$ in $P$. So  $\Sigma |c_{ijk}| = q$. 
According to the skein relation,  

$$M_P \leq 6|c_{1432}| + 4(|c_{1342}|+|c_{1243}|+|c_{1324}|+|c_{1423}|)+2|c_{1234}| +2|s_{max}|-2$$

$$m_P \geq  -2|c_{1432}| -4(|c_{1342}|+|c_{1243}|+|c_{1324}|+|c_{1423}|)-6|c_{1234}| -(2|s_{min}| -2)$$

Hence, \\

\begin{align}span(<K>) \leq M_P- m_P & \leq 8 \Sigma |c_{ijk}| + 2(|s_{max}| + |s_{min}|) - 4 \notag \\
 &= 8q+ 2(|s_{max}| + |s_{min}|) - 4 \notag
 \end{align}
 
 Thus, it suffices to prove that $|s_{max}| + |s_{min}| \leq 4q + 2$ for any connected quadruple crossing projection with $q$ quadruple crossings.

We apply induction on $q$. If there is a single quadruple crossing, the single splitting for $s_{max}$ is determined by the type of crossing, but in all cases the high splitting will either be  a U-splitting as in Figure \ref{fig:splitting1} or a parallel splitting as in Figure \ref{fig:splitting2} and \ref{fig:splitting3}, allowing for rotation. If the high splitting is one or two of these, the low splitting or splittings come from those that remain. Note that the relative orientations of these three splittings must appear as in Figure \ref{fig:highlowsplittings} but all three can be rotated together. 

\begin{figure}[h]
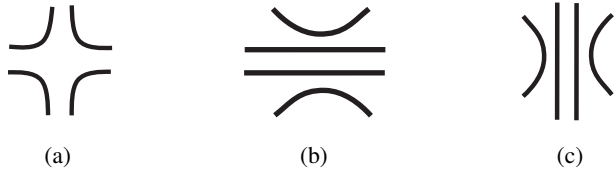

	\begin{subfigure}[b]{.2\textwidth}
		\centering
		{\includegraphics[height=18mm]{State1432-2.pdf}}
		\caption{}
		\label{fig:splitting1}		
	\end{subfigure} 
	\begin{subfigure}[b]{.2\textwidth}
		\centering
		{\includegraphics[height=18mm]{State1432-13.pdf}}
		\caption{}
		\label{fig:splitting2}
	\end{subfigure}
          \begin{subfigure}[b]{.2\textwidth}
		\centering
		{\includegraphics[height=18mm]{State1432-14.pdf}}
		\caption{}
		\label{fig:splitting3}		
	\end{subfigure}
	\\ \vspace{.2cm}
		\caption{Splittings that can correspond to high and low states, allowing rotation of all three together.}
		\label{fig:highlowsplittings}
\end{figure}
\vspace{.5cm}

It is then straightforward to check that if the high splitting yields 4 components, the low splitting must yield 2. If the high splitting yields 3 components, the low splitting must yield 1 or 3 components. If the high splitting yields 2 components, the low splitting must yield 2 or 4. And if the high splitting yields 1 component, the low splitting must yield 1 or 3. Thus, $|s_{max}| + |s_{min}| \leq 6$. 
 
Suppose now that we have proved that for a connected projection with $q$ quadruple crossings, $|s_{max}| + |s_{min}| \leq 4q + 2$. Given a connected quadruple crossing projection $P$ with $q+1$ quadruple crossings, choose one and split it as a high split. In the one case of a crossing of type $c_{1234}$ where there are two possible high splits, choose the one that together with all the other high splits would yield the greatest number of components. After splitting, we have a projection $P'$ of a link with $q$ quadruple crossings. This projection has 1,2,3 or 4 connected components. Suppose first that it has four connected components, labelled $L_1, L_2, L_3$ and $L_4$. Each of the  components contains some number of the quadruple crossings, denoted $q_1, q_2, q_3$ and $q_4$ respectively. We label their max and min states as $s'_{max,i}$ and $s'_{min,i}$.

Then by our induction hypothesis, $|s'_{max, i}| + |s'_{min, i}| \leq 4 q_i +2$ for each $i$. But $|s_{max}| = \sum |s'_{max, i}|$ since we split the first crossing as a high split. To obtain $|s_{min}|$, we take $|s'_{min,i}|$ for each component, but then we must change the first crossing we set from the high splitting to the low splitting. So the total number of components will go down by 2. Hence, $|s_{min}| = \sum _{i=1}^{4} |s'_{min, i}| -2$.Thus, $|s_{max}| + |s_{min}| \leq \sum _{i=1}^{4} (4 q_i +2) - 2 = 4q + 6 = 4(q+1) + 2$ as we wanted to show.

In the case that the high split at the first crossing yields three components, the low split at the same crossing yields at most three components. Hence,  $|s_{max}| + |s_{min}| \leq \sum _{i=1}^{3}(4 q_i +2)  = 4q + 6= 4(q+1) + 2$ as we wanted to show.

If the high split at the first crossing yields 2 components, the fact the low split can yield at most 4 components yields $|s_{max}| + |s_{min}| \leq \sum _{i+1}^{2}(4 q_i +2) + 2  = 4q + 6= 4(q+1) + 2$. Finally, if the high split yields only one component,  the low split can yield at most three  components so $|s_{max}| = |s'_{max}|$ and $|s_{min}| \leq |s'_{min}| +2$.  
Thus, we have  $|s_{max}| + |s_{min}| \leq  (4q +2) + 2 \leq 4(q+1) + 2$. Hence, in all cases, $|s_{max}| + |s_{min}| \leq  4(q+1) + 2$.
\end{proof}

\begin{corollary} \label{crossing} For any nontrivial alternating knot or link $K$, $q(K) \geq \frac{c(K)}{4}$.
\end{corollary}

\begin{proof} Results of \cite{Kau}, \cite{Muras} and \cite{Th} yield the fact that for a reduced alternating knot, \\ span$(<K>) = 4c[K]$. Hence, by Theorem \ref{span}, $4c(K) \leq 16 q(K)$, yielding the result.
\end{proof}

\begin{corollary} Let $T_{r,s}$ denote an $(r,s)$-torus knot. Then $q(T_{r,s}) \geq \frac{r+s-2}{4}$.
\end{corollary}

\begin{proof} This follows immediately from the fact that follows from results in \cite{Jones} that the span of the bracket polynomial for such a knot is $4(r+s-2)$.
\end{proof}

As we will see in the next section, this does yield the exact quadruple crossing number for the 2-braid knots, although we will not use this to prove it.

\section{Moves to obtain quadruple crossing number}

In this section, we exhibit in Figure \ref{fig:quadcrossings} a collection of moves that allows one to convert certain standard double crossing projections of a knot into quadruple crossing projections. With these moves and Theorem \ref{span} we can determine the quadruple crossing number of a variety of knots.
\begin{figure}[h]
	\begin{subfigure}[b]{.2\textwidth}
		\centering
		{\includegraphics[height=18mm]{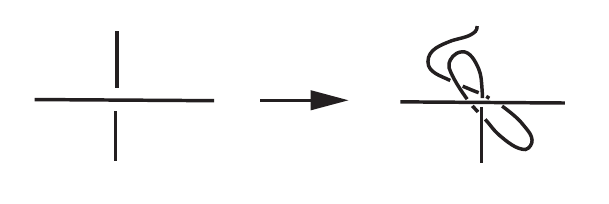}}
		\label{fig:quadmoveI}
		\caption{Move I (holds for any choice at crossing)}
	\end{subfigure} \\ \vspace{.5cm}
	\begin{subfigure}[b]{.2\textwidth}
		\centering
		{\includegraphics[height=20mm]{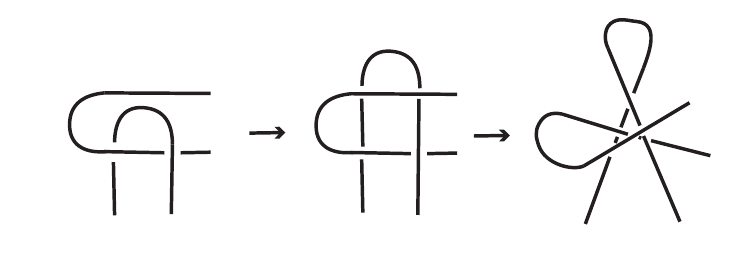}}
		\caption{Move II (holds for any choice at crossings)}
		\label{fig:1quadmoveII}
	\end{subfigure} \\ \vspace{.5cm}
          \begin{subfigure}[b]{.2\textwidth}
		\centering
		{\includegraphics[height=18mm]{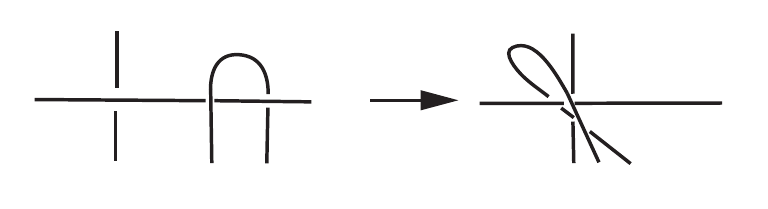}}
		\label{fig:quadmoveIII}
		\caption{Move III (holds for any choice at crossing)}
	\end{subfigure} \\ \vspace{.5cm}
	\begin{subfigure}[b]{.2\textwidth}
		\centering
		{\includegraphics[height=20mm]{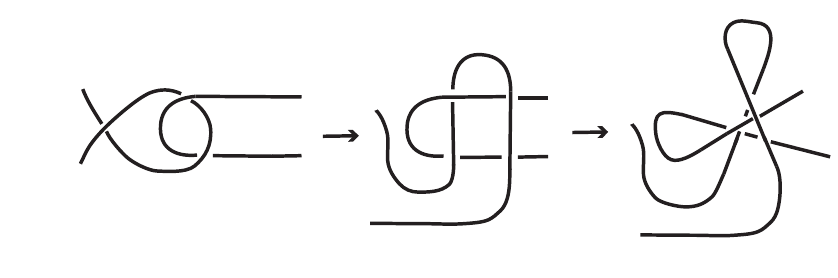}}
		\label{fig:quadmoveIV}
		\caption{Move IV}
	\end{subfigure}	 \\ \vspace{.5cm}
		\begin{subfigure}[b]{.2\textwidth}
		\centering
		{\includegraphics[height=20mm]{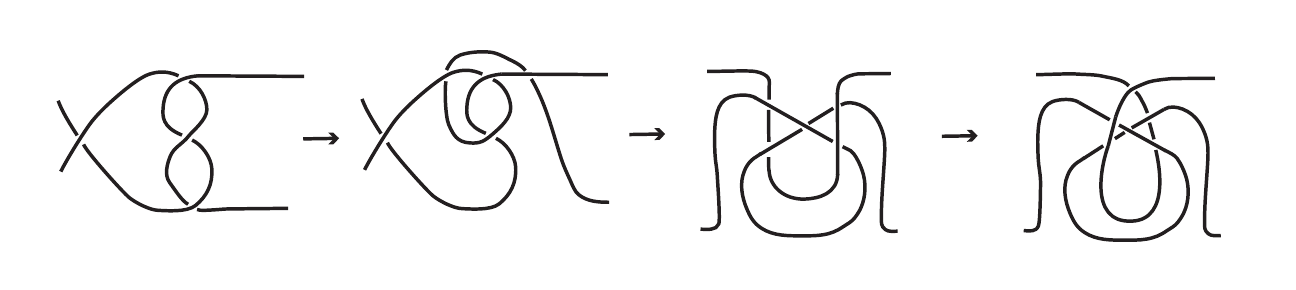}}
		\caption{Move V}
		\label{fig:quadmoveV}
	\end{subfigure} \\ \vspace{.5cm}
	\begin{subfigure}[b]{.2\textwidth}
		\centering
		{\includegraphics[height=20mm]{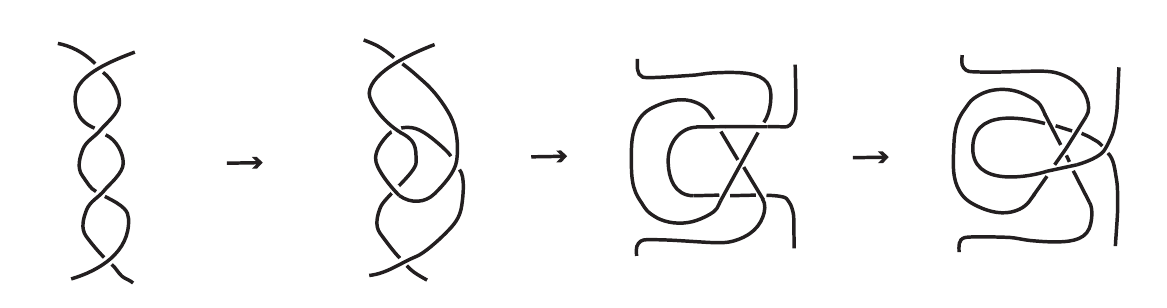}}
		\caption{Move VI}
		\label{fig:quadmoveVI}
	\end{subfigure} \\ \vspace{.5cm}
	\begin{subfigure}[b]{.2\textwidth}
		\centering
		{\includegraphics[height=20mm]{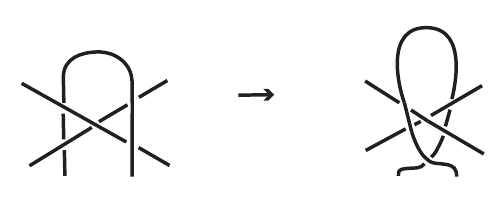}}
		\caption{Move VII}
		\label{fig:quadmoveVII}
	\end{subfigure}
	\vspace{.5cm}
	\caption{Moves to obtain quadruple crossings.}

\label{fig:quadcrossings}
\end{figure}
\vspace{.5cm}

Moves III and IV apply to regions of a projection containing 3 crossings. We call a region as appears on the left in those diagrams a 3-set. Note that Move IV is in fact a special case of Move III. Similarly Moves V and VI apply to regions of a projection containing 4 crossings. We call such a region a 4-set. A single crossing is a 1-set and an alternating bigon is a 2-set. In order to turn a standard projection into a quadruple crossing projection, we will be looking to subdivide it into 4-sets and 3-sets, with the possibility of a 1-set or 2-set as well.

\begin{corollary} \label{sets} Let $K$ be an alternating knot or link in a reduced alternating projection $P$ with $c$ crossings. 

 \begin{enumerate}
\item If $c = 4k$ for some $k \geq 0$, and a reduced alternating projection can be subdivided into 4-sets, then $q(K) = \frac{c}{4}$.
\item If $c = 4k+1$, and a reduced alternating projection can be subdivided into a disjoint collection of all 4-sets, and a 3-set and a 2-set, or  all 4-sets and a 1-set,  or all 4-sets and three 3-sets, then $q(K)= \frac{c+3}{4}$.
\item If $c= 4k+2$, and a reduced alternating projection can be subdivided into a disjoint collection of all 4-sets and a 2-set, or all 4-sets, and two 3-sets, then  $q(K)= \frac{c+2}{4}$.
\item If $c= 4k+3$, and  a reduced alternating projection can be subdivided into a disjoint collection of all 4-sets and one 3-set, then $q(K)=  \frac{c+1}{4}$.
\end{enumerate}
\end{corollary}

\begin{proof} By Theorem \ref{span}, we obtain a lower bound on quadruple crossing number which exactly matches the upper bound obtained by the number of sets in our subdivision of the projection, each set of which becomes a single quadruple crossing. 
\end{proof}

This corollary is relatively effective at determining the quadruple crossing number of small crossing alternating knots. (See Table \ref{quadtable}.) For instance, for the 5,6,7 and 9-crossing prime alternating knots, it determines exactly the quadruple crossing number for all except $7_6, 7_7, 9_{26},9_{31},9_{32}, 9_{33}, 9_{34}$, and $9_{40}$. It is less effective at determining the quadruple crossing number of the 8-crossing knots, only determining $q = 2$ for  $8_3$, $8_4$ and $8_9$.  However, by determining all knots generated by projections with $q=2$, and using Move VII on $8_{19}$, $8_{20}$ ad $8_{21}$, the quadruple crossing number of $7_6$ and $7_7$ and all 8-crossing knots except four is determined. Corollary \ref{sets} also allows us to determine that the quadruple crossing number  is 3 for many of the prime 10-crossing knots. Note that for $10_{33}$ and $10_{34}$, one must go to  non-alternating 11-crossing diagrams to obtain the necessary decomposition into sets. Any prime 10-crossing knot that does not appear in the table has unknown quadruple crossing number.

Corollary \ref{sets} applies to various infinite families of knots and links, including a variety of pretzel knots and rational knots. We mention two particularly simple families here.

\begin{corollary}
\begin{enumerate}
  \item Let $K_n$ be a 2-braid knot or link of $n$ crossings. Then $q(K_n) = \lceil \frac{n}{4} \rceil$.
\item Let $J_n$ be a rational knot with alternating Conway notation 3 a. Then $q(J_n) = \lceil \frac{3+a}{4} \rceil$.
\end{enumerate}
\end{corollary}

\begin{proof} In the case of $K_n$, the four cases of Corollary \ref{sets} yield the result. In the case of $J_n$, there is a 4-set coming from the 3 and the first crossing in the sequence of $a$ crossings. The remaining $a-1$ crossings can be either subdivided into all 4-sets, if $3+a = 4k$, or one 1-set and the rest 4-sets if $3+a=4k+1$, or one 2-set and the rest 4-sets if $3+a = 4k+2$ or one 3-set and the rest 4-sets if $3+a = 4k+3$.
\end{proof}

	\begin{table}[htbp!]
	\begin{center}
		
		\begin{tabular}{| c | c || c | c || c | c || c | c |}
			\hline
			Knot & $q(K)$ & Knot & $q(K)$ & Knot & $q(K)$& Knot & $q(K)$\\\hline
			$3_1$ & 1 & $9_8$ & 3 & $10_{1}$ & 3 & $10_{43}$ & ? \\\hline
			$4_1$ & 2 & $9_9$ & 3 & $10_{2}$ & 3  & $10_{44}$ & ? \\\hline
			$5_1$ & 2  &$9_10$ & 3 & $10_{3}$ & 3  & $10_{45}$ & ? \\\hline
			$5_{2}$ & 2 & $9_{11}$ & 3 & $10_{4}$ & 3  & $10_{46}$ & 3 \\\hline
			$6_1$ & 2  & $9_{12}$ & 3 & $10_{5}$ & 3  & $10_{47}$ & 3 \\\hline
			$6_{2}$ & 2  & $9_{13}$ & 3& $10_{6}$ & 3  & $10_{48}$ & 3 \\\hline
			$6_{3}$ & 2 & $9_{14}$ & 3 & $10_{7}$ & 3 & $10_{49}$ & 3 \\\hline
			$7_1$ & 2 & $9_{15}$ & 3 & $10_{8}$ & 3 & $10_{50}$ & 3 \\\hline
			$7_{2}$ & 2  & $9_{16}$ & 3 & $10_{9}$ & 3 & $10_{51}$ & 3   \\\hline
			$7_{3}$ & 2  & $9_{17}$ & 3 & $10_{10}$ & ? & $10_{52}$ & 3  \\\hline
			$7_{4}$ & 2 & $9_{18}$ & 3 & $10_{11}$ & 3 & $10_{53}$ & 3 \\\hline
			$7_{5}$ & 2  & $9_{19}$ & 3& $10_{12}$ & 3 & $10_{54}$ & ? \\\hline
			$7_{6}$ & 3  & $9_{20}$ & 3 & $10_{13}$ & 3 & $10_{55}$ & ? \\\hline
			$7_7$ &  3 & $9_{21}$ & 3 & $10_{14}$ & 3 & $10_{56}$ & ? \\\hline
			$8_1$ & 3 & $9_{22}$ & 3  & $10_{15}$ & 3 & $10_{57}$ & ? \\\hline
			$8_{2}$ & 3  & $9_{23}$ & 3 & $10_{16}$ & 3 & $10_{58}$ & ? \\\hline
			$8_{3}$ & 2  & $9_{24}$ & 3 & $10_{17}$ & ? & $10_{59}$ & ? \\\hline
			$8_{4}$ & 2  & $9_{25}$ & 3 & $10_{18}$ & 3 & $10_{60}$ & ? \\\hline
			$8_{5}$ & 3  & $9_{26}$ & ? & $10_{19}$ & 3 & $10_{61}$ & 3 \\\hline
			$8_{6}$ & 3  & $9_{27}$ & 3 & $10_{20}$ & 3 & $10_{62}$ & 3 \\\hline
			$8_7$ & 3 & $9_{28}$ & 3  & $10_{21}$ & 3 & $10_{63}$ & 3 \\\hline
			$8_8$ & 3  & $9_{29}$ & 3 & $10_{22}$ & 3 & $10_{64}$ & 3 \\\hline
			$8_9$ & 2  & $9_{30}$ & 3 & $10_{23}$ & 3 & $10_{65}$ & 3 \\\hline
			$8_{10}$& 3  & $9_{31}$ & ? & $10_{24}$ & 3  & $10_{66}$ & 3 \\\hline
			$8_{11}$ &3 & $9_{32}$ & ?& $10_{25}$ & 3 & $10_{74}$ & 3 \\\hline
			$8_{12}$& 3 & $9_{33}$ & ?& $10_{26}$ & 3 & $10_{76}$ & 3 \\\hline
			$8_{13}$ & ? & $9_{34}$ & ? & $10_{27}$ & ? & $10_{77}$ & 3 \\\hline
			$8_{14}$ &3  & $9_{35}$ & 3 & $10_{28}$ & 3 & $10_{124}$ & 3 \\\hline
			$8_{15}$ & 3 & $9_{36}$ & 3 & $10_{29}$ & 3 & $10_{125}$ & 3 \\\hline
			$8_{16}$ & ? & $9_{37}$ & 3 & $10_{30}$ & 3 & $10_{126}$ & 3 \\\hline
			$8_{17}$ & ? & $9_{38}$ & 3 & $10_{31}$ & 3  & $10_{127}$ & 3 \\\hline
			$8_{18}$ & ? & $9_{39}$ & 3 & $10_{32}$ & 3  & $10_{130}$ & 3 \\\hline
			$8_{19}$ & 2 & $9_{40}$ & ? & $10_{33}$ & 3  & $10_{131}$ & 3 \\\hline
			$8_{20}$ & 2 & $9_{41}$ & 3 & $10_{34}$ & 3  & $10_{134}$ & 3 \\\hline
			$8_{21}$ & 2 & $9_{42}$ & 3 & $10_{35}$ & 3  & $10_{135}$ & 3 \\\hline
			$9_1$ & 3 & $9_{43}$ & 3 & $10_{36}$ & 3  & $10_{139}$ & 3  \\\hline
			$9_{2}$ & 3 & $9_{44}$ & 3  & $10_{37}$ & ?   & $10_{140}$ & 3 \\\hline
			$9_{3}$ & 3  & $9_{45}$ & 3  & $10_{38}$ & ?   & $10_{142}$ & 3 \\\hline
			$9_{4}$ & 3  & $9_{46}$ & 2  & $10_{39}$ & ? & $10_{144}$ & 3  \\\hline
			$9_{5}$ & 3  & $9_{47}$ & ?  & $10_{40}$ & ? & $10_{142}$ & 3  \\\hline
			$9_{6}$ & 3  & $9_{48}$ & 3  & $10_{41}$ & ?  & $10_{148}$ & 3 \\\hline
			$9_7$   &  3  &  $9_{49}$ & 3 & $10_{42}$ & ?  &   $10_{161}$ & 3 \\\hline
					
				\end{tabular}

\vspace{.5cm}
\caption{Known quadruple crossing numbers.}
\label{quadtable}
\end{center}
\end{table}

\end{document}